\documentclass[10pt,twoside,a4]{amsart}
\usepackage{amssymb,amsmath,amsfonts,amsthm,graphicx,amscd}
\usepackage{enumerate,mathrsfs,longtable,bm,float,xcolor}
\usepackage[dvips]{epsfig}
\usepackage{amsmath,amsfonts,amssymb,amsthm,bm,float,makecell,booktabs,MnSymbol}
\usepackage{caption}
\usepackage{mathtools}

\usepackage{enumerate}
\usepackage{amscd}
\usepackage{tikz}
\usetikzlibrary{positioning}
\usepackage[all,cmtip,line]{xy}

\DeclareMathAlphabet{\mathpzc}{OT1}{pzc}{m}{it}

\numberwithin{equation}{section}

\makeatletter
\@namedef{subjclassname@2020}{Mathematics Subject Classification}
\makeatother

\theoremstyle{plain}

\newtheorem{lem}{Lemma}[section]

\newtheorem{thm}[lem]{Theorem}
\newtheorem{prop}[lem]{Proposition}
\newtheorem{cor}[lem]{Corollary}
\theoremstyle{definition}
\newtheorem{exa}[lem]{Example}
\newtheorem{rem}[lem]{Remark}

\newtheorem{defn}[lem]{Definition}

\newtheorem{que}[lem]{Question}

\begin{document}

\baselineskip 13truept

\title[$\mathfrak{X}$-elements in  multiplicative lattices ]{$\mathfrak{X}$-elements in  multiplicative lattices - A generalization of $J$-ideals, $n$-ideals and $r$-ideals in rings}

\author{Sachin Sarode* and Vinayak Joshi**}
\address{\rm *Department of Mathematics, Shri Muktanand College\\ Gangapur, Dist. Aurangabad - 431 109, India.} \email{sarodemaths@gmail.com}

\address{\rm **Department of Mathematics, Savitribai Phule Pune University,
	Pune-411 007, India.}
\email{vvjoshi@unipune.ac.in \\
	vinayakjoshi111@yahoo.com }

\date{January 17, 2021} \maketitle

\begin{abstract}
	In this paper, we introduce a concept of  $\mathfrak{X}$-element with respect to an $M$-closed set $\mathfrak{X}$ in multiplicative lattices and study properties of $\mathfrak{X}$-elements. For a particular $M$-closed subset $\mathfrak{X}$, we define the concept of $r$-element, $n$-element and $J$-element. These elements generalize the notion  of $r$-ideals, $n$-ideals and $J$-ideals of a commutative ring with unity to multiplicative lattices. In fact, we prove that an ideal $I$ of a commutative ring $R$ with unity is a $n$-ideal ($J$-ideal) of $R$ if and only if it is an $n$-element ($J$-element) of $Id(R)$, the ideal lattice of $R$.

\end{abstract}


\section{Introduction}

The ideal theory of commutative rings with unity is very rich. Many researchers defined different ideals ranging from prime ideals, maximal ideals, primary ideals to recently introduced $r$-ideals, $n$-ideals and $J$-ideals. More details about $r$-ideals, $n$-ideals, and $J$-ideals can be found in Mohamadian \cite{M}, Tekir et al. \cite{TKO} and, Khashan and Bani-Ata \cite{KB} respectively.

Ward and Dilworth \cite{wd} introduced the concept of multiplicative lattices to generalize the ideal theory of commutative rings with unity. Analogously, the concepts of a prime element, maximal element, primary elements are defined. 

The study of prime elements and its generalization in a multiplicative lattice is the main focus of many researchers. Different classes of elements and generalization of a prime element in multiplicative lattices were studied; see   Burton \cite{B}, Joshi and Ballal \cite{JB}, Jayaram \cite{J}, Jayaram and Johnson \cite{JJ, JJJ}, Jayaram et al. \cite{JTY}, Manjarekar and Bingi \cite{MB}. 

We observed a unifying pattern in the results of $J$-ideals, $n$-ideals and $r$-ideals of rings.  This motivates us to introduce a new class of elements, namely $\mathfrak{X}$-element in multiplicative lattices. Hence this study will unify many of the results proved for these ideals and generalize it to multiplicative lattice settings. Further, for a particular $M$-closed subset $\mathfrak{X}$, we define the concept of $r$-element, $n$-element and $J$-element. Hence this justifies the name $\mathfrak{X}$-element. These elements are the generalizations of $r$-ideals, $n$-ideals and $J$-ideals of a commutative ring with unity. In fact, we prove that an ideal $I$ of a commutative ring $R$ with unity is a $n$-ideal($J$-ideal) of $R$ if and only if it is an $n$-element ($J$-element) of $Id(R)$, the ideal lattice of $R$.

\pagebreak
Now, we begin with the necessary concepts and terminology.

\begin{defn} \label{1.1.}   A nonempty
	subset $I$ of a lattice $L$ is a \textit{semi-ideal} if $x \leq
	a\in I$ implies $x\in I$. A semi-ideal $I$ of $L$ is an
	\textit{ideal} if $a \vee b \in I$ whenever $a,b\in I$. An ideal
	(semi-ideal) $I$ of a lattice $L$ is a \textit{proper} ideal
	(semi-ideal) of $L$ if $I\neq L$.  A proper ideal (semi-ideal) $I$
	is \textit{prime} if $a\wedge b \in I$ implies $a \in I $ or $b
	\in I$, and it is \textit{minimal} if it does not properly contain
	another prime ideal (prime semi-ideal).  For $a\in L$, let
	$(a]=\{x \in L\colon\, x \leq a\}$. The set $(a]$ is the
	\textit{principal ideal generated by a}.

	A lattice $L$ is \textit{complete} if for any subset $S$ of $L$,
	we have $\bigvee S,\bigwedge S\in L$. The smallest element   and
	the greatest element  of a lattice $L$ is denoted by 0 and 1
	respectively.

	The concept of multiplicative lattices was introduced by Ward and Dilworth \cite{wd} to study the abstract commutative ideal theory of commutative rings.

	A complete lattice $L$ is a \textit{multiplicative
		lattice} if there exists a binary operation $``\cdot"$ called the
	\textit{multiplication} on $L$ satisfying the following
	conditions:
	\begin{enumerate} \item $a \cdot b=b \cdot a$, for all $a,b\in L$. \item
		$a \cdot(b \cdot c)=(a \cdot b)\cdot c$, for all $a,b,c\in L$. \item
		$a \cdot(\bigvee_{\alpha} b_{\alpha})=\bigvee_{\alpha}(a \cdot
		b_{\alpha})$, for all $a,b_{\alpha}\in L$, $\alpha \in \Lambda$(an
		index set).  \item $a \cdot 1=a$, for all $a\in L$.
	\end{enumerate}
	
	Note that in a multiplicative lattice $L$, $a\cdot b \leq a \wedge
	b$ for $a,b \in L$. For this, let $a=a\cdot 1= a \cdot (b \vee 1)=
	a\cdot b \vee a$. Thus $a\cdot b \leq a$. Similarly, $a\cdot b
	\leq b$. This proves that $a\cdot b \leq a\wedge b$. Moreover, if
	$a \leq b$ in $L$, then $a\cdot c \leq b \cdot c$ for every $c \in
	L$. Also, if $a \leq b$ and $c\leq d$ then $a \cdot c \leq b
	\cdot d$.
	
	An element $ c $ of a complete lattice $ L $ is \textit{compact}
	if $ c \leq \bigvee_{\alpha} a_{\alpha} $, $\alpha \in \Lambda$
	($\Lambda$ is an index set) implies $ c \leq \bigvee _{i=1}^{n}
	a_{\alpha_{i}} $, where $ n \in \mathbb{Z}^{+} $. 	The set of all compact elements of a lattice $ L $ is denoted by
	$L_{*} $.
	
	A lattice $L$ is \textit{compactly generated} or
	\textit{algebraic} if for every $ x \in L $, there exist $
	x_{\alpha} \in L_{*} $ for $\alpha \in \Lambda$(an index set)
	such that $ x = \bigvee_{\alpha} x_{\alpha} $, that is, every
	element is a join of compact elements. Equivalently, if $L$ is a
	compactly generated lattice and if $a\not\leq b$ for $a, b \in L$,
	then there exists a nonzero compact element $c\in L_*$ such that
	$c \leq a$ and $c \not\leq b$.
	
	A multiplicative lattice $L$ is \textit{$1$-compact} if $1$ is a compact element of $L$.
	A multiplicative lattice $L$ is \textit{compact} if every element of $L$
	is a compact element.
	
	A multiplicative lattice $L$ is a {\it $c$-lattice} if $L$ is
	1-compact, compactly generated multiplicative lattice in which the product of two compact elements is compact. Note that the ideal lattice of a commutative ring $R$ with unity is always a $c$-lattice.
	
	An element $p$ of a multiplicative lattice $L$ with $p \neq 1$ is \textit{prime} if $a
	\cdot b \leq p$ implies $a \leq p$ or $b \leq p$. It is not
	difficult to prove that an element $p$ (with $p \neq 1$) of a
	$c$-lattice $L$ is \textit{prime} if $a \cdot b \leq p$ for $a, b
	\in L_{*}$ implies $a \leq p$ or $b \leq p$. An element $p$ is
	said to be a \textit{minimal prime element} if there is no prime
	element $q$ such that $q <p$. An ideal $P$ of a commutative ring $R$ with unity is prime if and only if it is  a prime element of $Id(R)$, the ideal lattice of $R$.
	
	Let $L$ be a $c$-lattice and $a \in L$. Then the  \textit{radical}   of $a$ is denoted by $\sqrt{a}$ and given by \linebreak
	$\sqrt{a}= \bigvee \{x \in L_{*} \;|\; \ x^{n} \leq a$ for some $n
	\in \mathbb{N} \}$. Note that if any compact element $c \leq \sqrt{a}$, then $c^{m} \leq a$ for some $ m \in \mathbb{N} $. An element $a$
	of a $c$-lattice is a \textit{radical element},  if $a = \sqrt{a}$.
	A $c$-lattice is called a \textit{domain} if $0$ is a prime element of $L$.

	A proper element $i$ of a $c$-lattice $L$ is called \textit{ primary element} if whenever $a \cdot b  \leq i$ for some $a, b \in L$ then either $a \leq i$ or $b \leq \sqrt{i}$.
	

	A proper element $m$ of a multiplicative lattice is said to be \textit{maximal}, if $m \leq n <1$, then $m=n$. The set of all maximal elements of $L$ is denoted by Max$(L)$.
	The \textit{Jacobson radical} of $L$ is the set $J(L) = \displaystyle \bigwedge\{m\ | \ m \in \text{Max}(L)\}$.
	It is easy to observe that a maximal element of a $c$-lattice $L$ is prime. 
	A $c$-lattice $L$ is said to be \textit{local}, if $L$ has the unique maximal element $m$. In this case, we write $(L; m)$.
	
	A non-empty subset $\mathfrak{X}$ of $L_{*}$(set of all compact elements) in a $c$-lattice $L$ is \textit{multiplicatively closed} if
	$s_{1} \cdot s_{2} \in \mathfrak{X}$, whenever $s_{1}, s_{2} \in \mathfrak{X}$.
	
	A non-empty subset $\mathfrak{X}$ of a multiplicative lattice $L$ is  called \textit{$M$-closed} if $a, b \in \mathfrak{X}$, then  $a \cdot b \in \mathfrak{X}$.
	
	From the definitions, it is clear that every $M$-closed subset $A$ of a $c$-lattice is a multiplicatively closed subset of $L$.  The converse is not true. However, if $L$ is a compact lattice or finite, then $L=L_*$ and hence both definitions coincide with each other. Further, if $p$ is a prime element of a $c$-lattice $L$, then $L \setminus (p]$ is an $M$-closed subset of $L$.

	In a multiplicative lattice $L$, an element $a \in L$  is
	\textit{nilpotent} if $a^{n}= 0$ for some $ n \in
	\mathbb{Z}^{+}$, and $L$ is \textit{reduced} if the only nilpotent
	element is 0. The set of all nilpotent elements of a
	multiplicative lattice $L$ is denoted by  Nil$(L)$. 
	
	We denote the set $Z(L)$ of zero-divisors in $L$ by the set  $Z(L)=\bigl\{x\in L\ | \ x\cdot y=0 \text{ for some } y \in L\setminus\{0\}\bigr\}$. Clearly, Nil$(L) \subseteq Z(L)$.

	
	Let $L$ be a multiplicative lattice and $a,b \in L$. Then $(a:b) =
	\bigvee \{x \;|\;  x\cdot b \leq a\}$. Note that $x \cdot b \leq a
	\Leftrightarrow x \leq (a:b)$. Clearly, $a \leq (a:b)$  and
	$(a:b)\cdot b \leq a$ for  $a, b \in L$. 
	If $a\in L$, then $ann_L (a) =\bigvee\{x \in L \;|\: a \cdot x=0\}$. 
\end{defn}


For undefined concepts in lattices, see Gr\"atzer
\cite{G}.

\section{$\mathfrak{X}$-Elements in  multiplicative Lattices}
We introduce the concept of an $\mathfrak{X}$-element in multiplicative lattices.

\begin{defn}\label{4.1.} Let $L$ be a multiplicative lattice and $\mathfrak{X}$ be an $M$-closed subset of $L$. A proper element $i$ of a multiplicative lattice $L$ is called an {\it $\mathfrak{X}$-element}, if  $a \cdot b ~\leq i$ with $a ~\notin \mathfrak{X}$  implies $b~ \leq i$ for  $a, ~b ~\in L$.  \end{defn}

\vspace{-.2in}
\begin{minipage}{0.4\linewidth}
	\begin{exa} \label{4.2.}
		Consider a lattice $K$ whose Hasse diagram is shown in Figure
		\ref{40.1.}. On $K$, define the trivial multiplication $x \cdot y = 0 = y \cdot
		x$ for every $x, ~y \not\in \{1\} $ and $x \cdot 1 = x = 1 \cdot x$ for
		every $x \in K$. It is easy to see that $K$ is a multiplicative
		lattice. Moreover, $K$ is non-reduced. If we take $\mathfrak{X} = \{0,~ a,~ b, ~c, ~d\}$, then every proper element of $K$ is an $\mathfrak{X}$-element of $K$.
	\end{exa}
\end{minipage}
\begin{minipage}{0.6\linewidth}
	\begin{center}
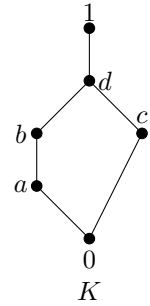

		\vspace{.3in}
		\begin{tikzpicture}[scale=.7]
			\draw (5,0) -- (4,1) node at (5, -0.4) {$0$}; \draw [fill=black] (5,0) circle
			(.1); \draw (5,0) -- (6,2) node at (6, 2.3) {$c$}; \draw [fill=black] (6,2)
			circle (.1); \draw (4,1) -- (4,2) node at (3.7, 1) {$a$}; \draw [fill=black]
			(4,1) circle (.1); \draw (4,2) -- (5,3) node at (3.7, 2) {$b$};
			\draw [fill=black] (4,2) circle (.1); \draw (6,2) -- (5,3) node at (5.3, 3)
			{$d$}; \draw [fill=black] (5,3) circle (.1); \draw (5,3) -- (5,4) node at
			(5,4.3) {$1$}; \draw [fill=black] (5,4) circle (.1); \draw node at (5,-1)
			{$K$};
		\end{tikzpicture}
		\captionof{figure}{A multiplicative lattice  in which every proper element is an $\mathfrak{X}$-element}\label{40.1.}
	\end{center}
\end{minipage}

\begin{rem} Note that a proper element of a multiplicative lattice  $L$ is an  $\mathfrak{X}$-element or not,  depends on an $M$-closed subset $\mathfrak{X}$ under consideration. 
	If $x$ is an $\mathfrak{X}_1$-element with respect to an $M$-closed subset $\mathfrak{X}_1$, then $x$ may or may not be  an $\mathfrak{X}_2$-element with respect to an $M$-closed subset $\mathfrak{X}_2$ different from $\mathfrak{X}_1$. 
	
	Also, note that if $L$ is a multiplicative lattice and $\mathfrak{X} = \{ 1 \}$ is an $M$-closed subset of $L$, then $L$ does not contain an $\mathfrak{X}$-element.    \end{rem}

\begin{lem}\label{4.10.} Let $L$ be a multiplicative lattice and $\mathfrak{X}$ be an $M$-closed subset of $L$. If $i$ is an $\mathfrak{X}$-element of $L$, then $(i] ~\subseteq \mathfrak{X}$. In particular, if $(i] = \mathfrak{X}$, then $i$ is an $\mathfrak{X}$-element of $L$ if and only if $i$ is a prime element of $L$. 
\end{lem}

\begin{proof} Suppose $i$ is an $\mathfrak{X}$-element of a multiplicative lattice $L$ and  let $x \in (i] $. Suppose on the contrary that $x \notin \mathfrak{X}$. Clearly, $ x \cdot 1 \leq i$ with $x \notin \mathfrak X $. Since $i$ is an $\mathfrak{X}$-element, we get $1 \leq i$, a contradiction to the fact that $i$ is a proper element of $L$. Therefore  $x \in \mathfrak{X}$ and hence $(i] ~\subseteq \mathfrak{X}$.

	Now, we prove ``in particular" part. Suppose that $(i] = \mathfrak{X}$ and $i$ is an $\mathfrak{X}$-element of $L$. Let $a, b \in L$ such that $a \cdot b \leq i$ and $a \nleq i$, i.e., $a \notin \mathfrak{X}$. As $i$ is an $\mathfrak{X}$-element, $b \leq i$. So $i$ is a prime element of $L$. 
	
	Conversely, suppose that $(i] = \mathfrak{X}$ and $i$ is a prime element of $L$. Let $a, b \in L$ such that  $a \cdot b \leq i$ with $a \not \in \mathfrak{X} = (i]$. By primeness of $i$, $b 
	\leq i$. Thus $i$ is an $\mathfrak{X}$-element of $L$. 
\end{proof}

\begin{rem}\label{4.11.}
	The converse of the Lemma \ref{4.10.} need not be true in general, i.e., if $i$ is a proper element of a multiplicative lattice $L$ such that $i \in \mathfrak{X}$, then $i$ need not be an $\mathfrak{X}$-element of $L$. Consider the ideal lattice $L$ of the ring $\mathbb{Z}_{12}$. Clearly, $L$ is a non-reduced lattice. Put $\mathfrak{X}= \{(0), ~(6)\}$. Then $(0),  (6) \in \mathfrak{X}$ but $(0)$ and $(6)$ are not  $\mathfrak{X}$-elements of $L$.  
\end{rem}

\begin{lem}\label{imply}
	Let $L$ be a multiplicative lattice and $\mathfrak{X}$ and $\mathfrak{X'}$  be $M$-closed subsets of $L$ such that $\mathfrak{X} \subseteq \mathfrak{X'}$. If $i$ is an $\mathfrak{X}$-element of $L$, then $i$ is an $\mathfrak{X'}$-element of $L$.
\end{lem}

\begin{proof}
	follows from the definition of an $\mathfrak{X}$-element.
\end{proof}

\begin{lem}\label{local}
	Let $(L;m)$ be a local lattice. Then every proper element of $L$ is an $\mathfrak{X}$-element for $\mathfrak{X}=(m]$. 
\end{lem}

\begin{proof}
	Let $a\cdot b \leq i$ and $a \notin \mathfrak{X}=(m]$. Since $L$ is local, $a=1$. Hence in this case $b \leq i$. This proves that $i$ is an $\mathfrak{X}$-element.  
\end{proof}

\begin{lem}\label{meet-irr}
	Assume that every proper element of a $c$-lattice $L$ is an $\mathfrak{X}$-element, where $\mathfrak{X}=(m]$ and $m \in L\setminus \{1\}$. Then $m$ is a unique maximal element of $L$. 
\end{lem}

\begin{proof}
	Let $i$ be a proper element of $L$ which is an $\mathfrak{X}$-element. Then by Lemma \ref{4.10.}, $i \leq m$. This is true for all proper elements $i$ of $L$. In particular, it is true for all maximal elements $m'$ too. This proves that $L$ has unique maximal element. Therefore it is a meet-irreducible element of $L$.
	\end{proof}

\begin{lem}\label{4.12.} Let $L$ be a multiplicative lattice and $\mathfrak{X}$ be an $M$-closed subset of $L$. If $\{i_j\}$, where $j \in \Lambda$ (an index set), is a non-empty set of  $\mathfrak{X}$-elements of $L$, then  $\bigwedge_j  i_j$ is also an $\mathfrak{X}$-element.  \end{lem}

\begin{proof}
	Obvious.\end{proof}

\begin{rem}\label{10.1.}
	The join of two $\mathfrak{X}$-elements is not necessarily an $\mathfrak{X}$-element. Consider the ideal lattice $L$ of $\mathbb{Z}_{15}$ with $\mathfrak X=\{(0), (3), (5)\}$. Then $(3),  (5)$ are $\mathfrak{X}$-elements of $L$, but $(3) \vee (5) = (1)$ is not an\mbox{$\mathfrak{X}$-element}. \end{rem}

\begin{lem}\label{4.13.}  Let $i$ be a proper element of a $c$-lattice $L$ and $\mathfrak{X}$ be an $M$-closed subset of $L$. Then $i$ is an $\mathfrak{X}$-element of $L$ if and only if $i = (i : a)$ for all $a \notin \mathfrak{X}$. In particular, if $i$ is an $\mathfrak{X}$-element of $L$, then $(i : a)$ is an $\mathfrak{X}$-element of $L$ for all $a \notin \mathfrak{X}$.\end{lem}

\begin{proof}
	Suppose $i$ is an $\mathfrak{X}$-element of $L$ and let  $a \notin \mathfrak{X}$. We always have $i \leq (i : a) $. Let $x$ be any compact element such that $x \leq (i : a)$. Therefore  $x \cdot a \leq i$. Since  $i$ is an $\mathfrak{X}$-element and $a \notin \mathfrak{X} $, we get $x \leq i$. Hence $(i : a) \leq i$, as $L$ is a $c$-lattice. Therefore $i = (i : a)$ for all $a \notin \mathfrak{X}$.
	
	Conversely, suppose that 
	$i = (i : a)$ for all $a \notin \mathfrak{X} $. Let $c, d \in L$ such that $c \cdot d ~\leq i$ with $c ~\notin \mathfrak{X}$. We claim that $ d \leq i$. Since  $c \cdot d ~\leq i$, we have $d \leq (i : c)$. As $c ~\notin \mathfrak{X}$, by the assumption  $(i : c) = i$, we have $d \leq i$. Therefore, $i$ is an $\mathfrak{X}$-element of $L$.
	Further, ``in particular part" is easy to observe.
\end{proof}

\begin{lem} \label{4.17.}  Let $i$ be a proper element of a multiplicative lattice $L$ and $\mathfrak{X}$ be an $M$-closed subset of $L$. Then the following statements are equivalent.
	
	\begin{enumerate}
		\item $i$ is an $\mathfrak{X}$-element of $L$.
		\item $(i : a)$ is an $\mathfrak{X}$-element of $L$ for every $a \nleq i$.
		\item $\Bigl((i:a)\Bigr] \subseteq  \mathfrak{X}$ for all $a \nleq i$.
	\end{enumerate}
  \end{lem}

\begin{proof} $(1) \implies (2)$:  Suppose that $i$ is an $\mathfrak{X}$-element of  $L$ with $j \nleq i$. Clearly, $(i : j) \not = 1$. Let $a, b \in L$ such that $ a \cdot b \leq (i : j)$ with $a \notin \mathfrak{X}$. So $a \cdot b \cdot j \leq i$. As $i$ is an $\mathfrak{X}$-element and $a \notin \mathfrak{X}$, we get $ b \cdot j \leq i$, i.e., $b \leq (i : j)$. Therefore $(i : j)$ is an $\mathfrak{X}$-element of $L$.

$(2) \implies (3)$: follows from Lemma \ref{4.10.}.

$(3) \implies (1)$: 
Suppose that  $((i : a)] \subseteq \mathfrak{X} $ for all $a \nleq i$. Let $c, d \in L$ such that $c \cdot d \leq i$ with $c \notin \mathfrak{X} $. We claim that $d \leq i$. Suppose $d \nleq i$. So by the assumption and $c \leq (i : d) $, we have $c\in \mathfrak{X}$, a contradiction. Therefore $d \leq i$. Hence $i$ is an $\mathfrak{X}$-element of $L$.	 
\end{proof}

\begin{lem} \label{4.18.} Let $L$ be a multiplicative lattice  and $\mathfrak{X}$ be an $M$-closed subset of $L$. If $i$ is a maximal $\mathfrak{X}$-element of $L$, then $i$ is a prime element of $L$.  \end{lem}

\begin{proof} Suppose $i$ is a maximal $\mathfrak{X}$-element of $L$. Let $a, b \in L$ such that $ a \cdot b \leq i $ and  $a \nleq i$. Since $i$ is an $\mathfrak{X}$-element and $a \nleq i$, by Lemma \ref{4.17.}, $(i : a)$ is an $\mathfrak{X}$-element of $L$. As $i$ is a maximal $\mathfrak{X}$-element of $L$ and $i \leq (i : a)$, we get $ (i : a) = i$. Therefore $b \leq i$.  
\end{proof}

\begin{lem}\label{2.8.} Let $j$ be a proper element of a $c$-lattice $L$ and  $\mathfrak{X} = (j]$  be an $M$-closed subset of $L$. Then   a proper element $i$ is  an $\mathfrak{X}$-element if and only if the condition $(*)$:
	
	$(*)$: for all $a, ~b ~\in L_*$ (set of all compact elements), $a \cdot b ~\leq i$ with $a ~\notin \mathfrak{X}$  implies $b~ \leq i$. 
\end{lem}

\begin{proof} 
	
	Assume that the condition $(*)$ holds. Let $a, b \in L$ such that $a \cdot b ~\leq i$ with $a ~\notin \mathfrak{X}$.  As $L$ is a $c$-lattice and $a ~\nleq j$, there exists $(0\not=) x \in L_*$ such that $x \leq a$ and $x ~\nleq j$. Now, let $y$ be a compact element such that $ y \leq b$. As $x \cdot y \leq a \cdot b \leq  i$ with $x ~\notin \mathfrak{X}=(j]$, by the condition $(*)$,  $ y \leq i$. Thus  every compact element $\leq b$ is  $\leq i$ and $L$ is a $c$-lattice, we get $b \leq i$. Hence $i$ is an $\mathfrak{X}$-element. 	The converse is obvious.
\end{proof}

\begin{lem}\label{a}
	Let $L$ be a $c$-lattice and $\mathfrak{X}=(j]$ be an $M$-closed subset of $L$.
	Then for a prime element $i$ of $L$ with $j \leq i$, $i$ is an $\mathfrak{X}$-element if and only if $i=j$. 
\end{lem}

\begin{proof}
	Assume that $i$ is a prime element which  is also an $\mathfrak{X}$-element of $L$.  By Lemma   \ref{4.10.}, we have $i \leq j$. This together with $j \leq i$, we have $i=j$. Conversely, assume that $i=j$ and $i$ is prime. To prove $i$ is an $\mathfrak{X}$-element, assume that $a\cdot b \leq i$ and $a \notin \mathfrak{X}$. Then by primeness of $i$ and $i=j$, we have $b \leq i$. 
\end{proof}

\begin{cor}\label{cor-a}
	Let $L$ be a $c$-lattice and $\mathfrak{X}=(j]$ be an $M$-closed subset of $L$. Then for a maximal element $i$ of $L$, $i$ is an $\mathfrak{X}$-element if and only if $i=j$. 
\end{cor}

\begin{proof}
	Assume that $i$ is a maximal element which  is also an $\mathfrak{X}$-element of $L$.  By Lemma   \ref{4.10.}, we have $i \leq j$. Thus by the maximality of $i$, we have $i=j$. Conversely, assume that $i=j$. Since $i$ is maximal, it is prime. Thus the result follows from Lemma \ref{a}. \end{proof}


\begin{thm} \label{2.20.}
	Let L be a $c$-lattice and  $\mathfrak{X}=(j]$ be an $M$-closed subset of $L$, where \\$j = \bigwedge P$ where $P=\bigl\{i_k\ | \ i_k \text{ is a prime elements of  } L\bigr\}$. Then the following statements are equivalent.
	
	\begin{enumerate}
		\item There exists an $\mathfrak{X}$-element in $L$.
		\item $j$ is a prime element of $L$.
		
	\end{enumerate}
	Moreover, if the set $Min(L)$ of all minimal prime elements in $L$  is finite, then all the above conditions are equivalent to $(*)$: $|Min(L)|=1$.
\end{thm}

\begin{proof} (1) $\Rightarrow$ (2):	Suppose there exists an $\mathfrak{X}$-element $i$ in $L$. Let $ \beta = \{x \ | \ x$ is an $\mathfrak{X}$-element in $L\} $. As $i \in \beta$, $\beta$ is a poset under induced partial order of $L$. Let $j_1 \leq j_2 \leq \cdots \leq j_n \leq \cdots $ be a chain $\mathcal{C}$ in $\beta$. We claim that $j = \bigvee^\infty_{ \alpha = 1} j_\alpha$ is in $\beta$, i.e., $j = \bigvee^\infty_{ \alpha = 1} j_\alpha$ is an $\mathfrak{X}$-element of $L$. Let $a, b \in L_*$ (set of all compact element of $L$) such that $a \cdot b \leq j = \bigvee^\infty_{ \alpha = 1} j_\alpha$ and $a \not\leq j$. As $L$ is a $c$-lattice  and $a, b \in L_*$, we get $a \cdot b \in L_*$. Therefore $a \cdot b \leq j_1 \vee j_2 \vee\cdots \vee j_n$ for some $j_1, j_2, \cdots , j_n \in  \mathcal{C}$. Since $\mathcal{C} $ is a chain, we must have $j_1 \vee j_2 \vee\cdots \vee j_n = j_\gamma$ for some $\gamma$, where $1 \leq \gamma \leq n$. Thus $a \cdot b \leq j_\gamma$ with $a \nleq j_\gamma$. Since $j_\gamma$ is an $\mathfrak{X}$-element of $L$, we have $b \leq j_\gamma$. Thus $b \leq  \bigvee^\infty_{ \alpha = 1} j_\alpha = j$. Therefore $j$ is an $\mathfrak{X}$-element of $L$. By Zorn's Lemma $\beta$ has a maximal element $w$, that is, $w$ is a maximal $\mathfrak{X}$-element. By Lemma \ref{4.18.}, $w$ is a prime element of $L$, that is, $w \in P$. Hence $j \leq w$.
	Also, $w$ is an $\mathfrak{X}$-element, we have by Lemma \ref{4.10.}, $w \leq j$. Thus $j=w$. Hence $j$ is prime.	
	
	(2) $\Rightarrow$ (1): Let $j$ be a prime element. By Lemma \ref{a}, $j$ is an $\mathfrak{X}$-element.

	\vskip5pt 
	We, now, prove $(2) \iff (*)$.
	
	Assume that $j$ is a prime element.  Since $j = \bigwedge\bigl\{i_k\ | \ i_k \text{ is a prime elements of  } L\bigr\}$ and the set $Min(L)$ is finite, without loss of generality, we assume that $\displaystyle j = \bigwedge_{k=1}^n\bigl\{i_k\ | \ i_k \in Min(L)\bigr\}$. By primeness of $j$ and $j \leq i_k$ for all $k$ with $i_k\in Min(L)$, we have $j=i_k$ for all $k$. Thus $|Min(L)|=1$.  
	
	Conversely, assume that 
	$|Min(L)|= 1$. Let $p$ be the only minimal prime element in $L$. Then $p \leq i_k$ for every $k$. Hence $j=p$. This proves that $j$ is prime.
\end{proof}

\begin{lem}[{L. Fuchs and R. Reis \cite[Lemma 2.5]{FR}}]\label{2.19.}
	Let $L$ be a $c$-lattice and $a \in L$. Then radical of $a$ is
	given by $\sqrt{a}= \bigwedge \{ p \in L \colon\ $p is  a minimal prime element over $a\}$.\end{lem}

\begin{lem} \label{2.22.} Let L be a $c$-lattice and  $\mathfrak{X}=(j]$ be an $M$-closed subset of $L$, where \\$j = \bigwedge\bigl\{i_k\ | \ i_k \text{ is a prime element of  } L\bigr\}$. Then a proper element $i$  is an $\mathfrak{X}$-element of $L$ if and only if $i$ is a primary element of $L$ and $\sqrt{i} = j$.
\end{lem}

\begin{proof} Suppose $i$ is an $\mathfrak{X}$-element of $L$.   By Lemma \ref{4.10.}, $(i] \subseteq \mathfrak{X}$. Hence $i \leq \sqrt{i} \leq \sqrt{j}=j$. Clearly, by Lemma \ref{2.19.},   $j \leq \sqrt{i}$. Hence $j = \sqrt{i}$. Let $a, b \in L$ such that $a \cdot b \leq i$. By Theorem \ref{2.20.}, $j$ is a prime element of $L$. Hence either $a \leq j$ or $b \leq j$. So either $a \leq \sqrt{i}$ or $b \leq \sqrt{i}$. Hence  $i$ is a primary element of $L$. 
	Conversely, suppose that  $i$ is a primary element of $L$ and $\sqrt{i}  = j$. Let $a, b \in L$ such that $a \cdot b \leq i$ and $a \nleq j$. Since $i$ is a primary element and $a \nleq j = \sqrt{i}$, we get $b \leq i$. Thus $i$ is an $\mathfrak{X}$-element of $L$. \end{proof}

\begin{lem} \label{2.23.} Let L be a $c$-lattice and  $\mathfrak{X}=(j]$ be an $M$-closed subset of $L$, where \\$j = \bigwedge\bigl\{i_k\ | \ i_k \text{ is a maximal element of  } L\bigr\}$. Then a proper element $i$  is an $\mathfrak{X}$-element of $L$ if and only if $i$ satisfies the following condition
	
	$(*)$: If $a\cdot b \leq i$, then $a \leq i$ or $b \leq m$, where $m= \bigwedge\bigl\{i_k\ | \ i_k \text{ is a maximal element\;}  \geq i~\bigr\}$ and $m = j$.
\end{lem}
\begin{proof}
	follows on similar lines as that of Lemma \ref{2.22.}.
\end{proof}

\begin{lem}\label{cancellation}
	Let $L$ be a $c$-lattice and $\mathfrak{X}$ be an $M$-closed subset of $L$. Let $k$ be an element of $L$ such that $k \notin \mathfrak{X}$. If $i_1$ and $i_2$ are $\mathfrak{X}$-elements with $i_1k=i_2k$, then $i_1=i_2$. Further, if $i$ is an element such that $ik$ is an $\mathfrak{X}$-element, then $ik=i$. 
\end{lem}

\begin{proof}
	Clearly, $i_1k \leq i_2$ with $k \notin \mathfrak{X}$. Since $i_2$ is an $\mathfrak{X}$-element, we have $i_1 \leq i_2$. On similar lines we can prove that $i_2 \leq i_1$. Thus $i_1=i_2$. 	Now, we prove ``further" part. Since $ik$ is an $\mathfrak{X}$-element and $ik \leq ik$ with $k \notin \mathfrak{X}$, we have $i \leq ik$. The reverse inequality is always true. Hence $i=ik$.
\end{proof}

It is well-known that if a proper ideal $P$ of a commutative ring $R$ with unity is prime if and only if $R\setminus P$ is a multiplicatively closed subset of $R$. Analogously, a proper element $p$ of a $c$-lattice $L$ is prime if and only if $L \setminus (p]$ is an $M$-closed subset of $L$. To characterize $\mathfrak{X}$-element, we define $\mathfrak{X}$-multiplicatively closed subset of $L$ as follows.

\begin{defn}\label{4.28.} Let  $\mathfrak{X}$ be an $M$-closed subset of a $c$-lattice $L$. A non-empty subset $A$ of $L_*$ with $(L_* \setminus \mathfrak{X} ) \subseteq  A$ is called a \textit{$\mathfrak{X}$-multiplicatively closed subset} of $L$, if $a_1 \in (L_* \setminus \mathfrak{X})$ and $a_2 \in A$, then $a_1 \cdot a_2 \in A$.
\end{defn}

\begin{rem} \label{6.1.} If $A$ is  a $\mathfrak{X}$-multiplicatively closed subset of $L$, then $A$ need not be a multiplicatively closed subset of $L$. Consider a multiplicative lattice $K$ given in Example \ref{4.2.}. If we take $\mathfrak{X} = \{0,~ a,~ b, ~c, ~d\}$, then $ A = \{ 1,~c,~ d \}$ is an $\mathfrak{X}$-multiplicatively closed subset of $K$ but  $A = \{ 1,~c,~ d \}$ is not a multiplicatively closed subset of $K$, as $c, d \in A$ and $c \cdot d = 0 \notin A $.

	Also, if $A$ is a multiplicatively closed subset of $L$, then $A$ need not be a $\mathfrak{X}$-multiplicatively closed subset of $L$ for some $M$-closed subset $\mathfrak{X}$. Consider the ideal lattice $L$ of the ring $\mathbb{Z}_{12}$. Then  $\{(1) \}$  is a multiplicatively closed subset of $L$, but $\{ (1) \}$ is not a  $\mathfrak{X}$-multiplicatively closed subset of $L$ for $\mathfrak{X} = \{(0), ~(6) \}$.
	
\end{rem}

\begin{lem}\label{4.29.}  Let $i$ be a proper element of a  $c$-lattice $L$ and $\mathfrak{X}$ be an $M$-closed subset of $L$. If $i$ is an $\mathfrak{X}$-element of $L$, then $L_* \setminus (i]$ is an $\mathfrak{X}$-multiplicatively closed subset of $L$. The converse is true if either $\mathfrak{X} = (j]$ is an $M$-closed subset of a $c$-lattice $L$ or $L$ is a compact lattice.
\end{lem}

\begin{proof} Suppose that $i$ is an $\mathfrak{X}$-element of $L$. By Lemma \ref{4.10.}, $(L_* \setminus \mathfrak{X}) \subseteq (L_* \setminus (i])$. Let $a \in (L_* \setminus \mathfrak{X})$ and  $b \in (L_* \setminus (i])$. We claim that $a \cdot b \in (L_* \setminus (i]) $. Suppose on the contrary that $a \cdot b \notin (L_* \setminus (i])$. So $ a \cdot b \leq i$. Since   $i$ is an $\mathfrak{X}$-element of $L$ and $a \notin \mathfrak{X}$, we get $b \leq i$, a contradiction to $b \in (L_* \setminus (i])$. Thus $a \cdot b \in (L_* \setminus (i]) $. Consequently, $L_* \setminus (i]$ is an $\mathfrak{X}$-multiplicatively closed subset of $L$.

	Conversely, suppose that $\mathfrak{X} = (j]$ is an $M$-closed subset of a $c$-lattice $L$ and  $L_* \setminus (i]$ is an \linebreak $\mathfrak{X}$-multiplicatively closed subset of $L$. Therefore $\bigl(L_* \setminus (j]\bigr) \subseteq \bigl(L_* \setminus (i]\bigr)$. In view of Lemma \ref{2.8.}, to show that $i$ is an $\mathfrak{X}$-element of $L$, it is enough to show that for $a, b \in L_*$ such that $a \cdot b \leq i$ with $a \not \in \mathfrak{X}=(j]$,  we have $b \leq i$. If $b \in (L_* \setminus(i])$, then as $(L_* \setminus (i])$ is an $\mathfrak{X}$-multiplicatively closed subset of $L$, we get $a \cdot b \in (L_* \setminus (i])$, a contradiction to $a \cdot b \leq i$. Therefore $i$ is an $\mathfrak{X}$-element of $L$.
	
	If $L$ is a compact lattice and $\mathfrak{X}$ be an $M$-closed subset of $L$, then the converse follows similarly.	
\end{proof}

\begin{thm}\label{2.30.} 
	Let $j$ be a proper element of $c$-lattice $L$ and  $\mathfrak{X} = (j]$  be an $M$-closed subset of $L$.  Suppose $a \in L$ and $ t \nleq
	a$ for all $t\in A$, where $A$ is an $\mathfrak{X}$-multiplicatively closed subset. Then there is an $\mathfrak{X}$-element $i$ of $L$ such that $a \leq
	i$ and $i$ is maximal with respect to $t \nleq i$ for all $t \in A$.
\end{thm}

\begin{proof} Let $R = \{c \in L \; | \ a \leq c \text{ and } \; t \nleq c$ for all $t \in A \}$. Clearly, $a \in R$ and hence $R$ is a poset under the induced partial order of $L$. Let $\mathcal{C}$ be a chain in $R$ and $w = \bigvee \{d \ | \ d \in \mathcal{C}\}$. We claim that $w \in R$. Suppose on the contrary that $w \notin R$, that is, $t \leq w$ for some $t \in A$. Since  $t\in A \subseteq L_*$ is compact, we have $t \leq d_1 \vee d_2 \vee \dots \vee d_n$ for some $d_1, d_2,  \cdots,  d_n \in \mathcal{C}$. As $\mathcal{C}$ is a chain we must have $d_1 \vee d_2 \vee \dots \vee d_n = d_i$ for some $i$,  where $1 \leq i \leq n$. Thus $t \leq d_i$, a contradiction. Thus $w \in R$. Hence by Zorn's lemma, there is a maximal element $i$ of $R$. Hence $a \leq i$ and $t \not\leq a$ for all $t \in A$.
	
	In view of Lemma \ref{2.8.}, to prove $i$ is an  $\mathfrak{X}$-element of $L$, assume that  $x , y \in L_*$ such that $ x \cdot y \leq i$  with $x \notin \mathfrak{X}=(j]$. Suppose that  $y \nleq j$. Clearly,  $y \leq (i : x)$ and, if $i = (i : x)$, then $y \leq i$ and we are done. Hence assume that $i < (i:x)$. Since $i$ is a maximal element of $R$, $(i:x) \not \in R$. Hence, there exists a compact element $t_1 \in A$ such that  $t_1 \leq ( i : x)$, that is, $x \cdot t_1 \leq i$. But $x \cdot t_1 \in A$, as $A$ is an $\mathfrak{X}$-multiplicatively closed subset, $x \in (L_* \setminus (j])$ and $t_1 \in A$. Thus there exist an element $t_2 = x \cdot t_1\in A$ such that $t_2  \leq i$,  a contradiction to $i \in R$. Hence $i$ is an $\mathfrak{X}$-element of $L$.\end{proof}

\begin{thm}\label{9.30.} 
	Let $L$ be a compact lattice and $\mathfrak{X}$ be an $M$-closed subset of $L$. Suppose $a \in L$ and $ t \nleq
	a$ for all $t\in A$, where $A$ is an $\mathfrak{X}$-multiplicatively closed subset. Then there is an $\mathfrak{X}$-element $i$ of $L$ such that $a \leq
	i$ and $i$ is maximal with respect to $t \nleq i$ for all $t \in A$.
\end{thm}

\begin{proof} Proof follows on similar lines as that of Theorem \ref{2.30.}. \end{proof}

\section{Applications of  $\mathfrak{X}$-Elements}

As already mentioned in the introduction, there is a unifying pattern in the results of $J$-ideals, $n$-ideals and $r$-ideals of a commutative ring with unity.

In this section, we prove these results of $J$-ideals, $n$-ideals and $r$-ideals  by suitably replacing the set $\mathfrak{X}$ in multiplicative lattices. Hence most of the results of the papers \cite{KB}, \cite{M} and \cite{TKO} becomes the corollaries of our results. 

First, we quote the definitions of 	$J$-ideals, $n$-ideals and $r$-ideals using the sets $Z(R)$, $N(R)$ and $J(R)$ ($Z(L)$, $N(L)$ and $J(L)$), the set of zero-divisors, the nil-radical  and the Jacobson radical  of a commutative ring $R$ (multiplicative lattice $L$) respectively.

\begin{defn}\label{8.1.}  A proper ideal $I$ of a commutative ring $R$ with unity  is called: 
	\begin{itemize}
		\item $r$-ideal, if  $ab \in I$ with $ann_R(a)=(0)$  implies $b~ \in I$ for all $a, ~b ~\in R$ (see Mohamadian \cite{M}).
		
		\item $n$-ideal, if  $ab \in I$ with $a ~\not \in \sqrt{0}$  implies $b~ \in I$ for all $a, ~b ~\in R$ (see U. Tekir et al. \cite{TKO}).
		
		\item $J$-ideal, if  $ab \in I$ with $a ~\not \in J(R)$  implies $b~ \in I$ for all $a, ~b ~\in R$ (see Khashan and  Bani-Ata \cite{KB}).
	\end{itemize}
\end{defn}

Analogously, we define the concepts of $r$-element, $n$-element and $J$-element in multiplicative lattices.	

\begin{defn}\label{8.1.}  A proper element $i$ of a multiplicative lattice $L$ is called: 
	\begin{itemize}
		\item $r$-element, if  $a \cdot b ~\leq i$ with $a ~\not \in Z(L)$  implies $b~ \leq i$ for all $a, ~b ~\in L_*$.
		
		\item $n$-element, if  $a \cdot b ~\leq i$ with $a ~\notin \bigl(\sqrt{0}\ \!\bigr]$  implies $b~ \leq i$ for all $a, ~b ~\in L_*$.
		
		\item $J$-element, if  $a \cdot b ~\leq i$ with $a ~\not \in \bigl(J(L)\bigr]$  implies $b~ \leq i$ for all $a, ~b ~\in L_*$.
	\end{itemize}
\end{defn}

We quote the following three results to prove that a proper ideal $I$ of a commutative ring $R$ with unity is an $r$-ideal, $n$-ideal and $J$-ideal if and only if it is an $r$-element, $n$-element and $J$-element of the multiplicative lattice $Id(R)$, the set of all ideals of $R$, respectively.

\begin{thm}[{Mohamadian \cite[Lemma 2.5]{M}}] \label{8.11.}
	Let $R$ be a commutative ring with unity and $I$ be a proper ideal of $R$. Then 
	$I$ is an $r$-ideal if and only if whenever $J$ and $K$ are ideals of $R$ with $J \not\subseteq Z(R) $ and $JK \subseteq I$, then $K \subseteq I$.
\end{thm} 

\begin{thm}[{U. Tekir et al. \cite[Theorem 2.7]{TKO}}] \label{6.11.}
	Let $R$ be a commutative ring with unity  and $I$ a proper ideal of $R$. Then the following are equivalent:
	\begin{enumerate}
		\item $I$ is an $n$-ideal of $R$.
		\item $I=(I:a)$ for every $a \notin \sqrt {0} $.
		\item For ideals $J$ and $K$ of $R$, $JK \subseteq I$ with $J \cap (R - \sqrt {0}) \not= \emptyset$ implies $K \subseteq I$.
	\end{enumerate}
\end{thm}

\begin{thm}[{H. A. Khashan et al. \cite[Proposition 2.10]{KB}}] \label{7.11.}
	Let $R$ be a commutative ring with unity  and $I$ a proper ideal of $R$. Then the following are equivalent:
	\begin{enumerate}
		\item $I$ is a $J$-ideal of $R$.
		\item $I=(I:a)$ for every $a \not \in J(R)$.
		\item For ideals $A$ and $B$ of $R$, $AB \subseteq I$ with $A \nsubseteq J(R)$ implies $B \subseteq I$.
	\end{enumerate}
\end{thm}

\begin{thm} \label{8.36} Let $R$ be a Noetherian ring with unity. Then $I$ is an $r$-ideal of $R$ if and only if $I$ is an $r$-element of the multiplicative lattice $L= Id(R)$, where $Id(R)$ is the  ideal lattice of $R$. \end{thm}

\begin{proof}		
	
	Suppose that $I$ is an $r$-ideal of $R$. Let $J, K$ be any ideals of $R$ such that $J \cdot K \leq I$ in  $L$ with $J \notin Z(L)$, that is, $ann_L(J) = 0_L$, where $(0_R)$ is the least element of $L$, denoted by $0_L$. We claim that  $ann_R(J)  =  (0_R)$. Suppose on the contrary that $(0_R\not =  ) x \in ann_R(J) $. Hence $(x)J=(0_R)$, a contradiction to the $ann_L(J)=0_L$. Hence $ann_R(J)=(0_R)$. 
	
	Now, we prove that $J \not\subseteq Z(R)$. Suppose on the contrary that $J \subseteq Z(R)$. Since $R$ is Noetherian,  $J \subseteq  Z(R)=\bigcup_{i=1}^{n}P_i$, where $P_i$'s are associate primes. By Prime Avoidance Theorem $J \subseteq P_k$ for some $k$. Since $P_k$ is an associated prime, we have $P_k=0:x$ for some $x\in R$. But this will contradicts the fact that $ann_R(J)=(0)$. Hence $J \not\subseteq Z(R)$. By Theorem \ref{8.11.}, $K \subseteq I$, i.e., $K \leq I$. Thus $I$ is an $r$-element of $L$.

	Conversely, suppose that  $I$ is a $r$-element of $L$. Let $a, b \in R$ such that $a \cdot b \in I$ with $ann_R (a) = (0_R)$. We claim that $b \in I$. Since  $(a \cdot b) = (a) \cdot (b) \subseteq I$, we have $a' \cdot b' \leq I$ in $L$, where $a' = (a),\ b' = (b)$.  Clearly, $a' \notin Z(L)$. Hence $b' \leq I$, i.e.,  $b \in I$. Thus $I$ is an $r$-ideal of $R$.  \end{proof}
\begin{rem}
	From the proof of Theorem \ref{8.36}, it is clear that every $r$-element of $Id(R)$ is an $r$-ideal of $R$. However, for the converse we need the assumption that a ring is Noetherian. It should be noted that  if we replace ``Noetherian ring" by ``ring satisfies strongly annihilator condition" still the result is true. By strong annihilator condition, we mean, for given ideal $I$ of $R$, there exists $a \in I$ such that $ann_R(I)=ann_R(a)$. 
	
	Further, we are unable to find an example to show that the condition that the ring is  Noetherian or satisfies strongly annihilator condition is necessary to prove the above Theorem \ref{8.36}. Hence we raise the following question.
\end{rem}

\begin{que}\label{que}
	Let $I$ be an $r$-ideal of a commutative ring $R$ with unity. Is $I$ an $r$-element of $Id(R)$? 
\end{que}

\begin{thm} \label{2.36} Let $R$ be a commutative ring with unity. Then $I$ is an $n$-ideal of $R$ with unity if and only if $I$ is a $n$-element of multiplicative lattice $L = Id(R)$, where $Id(R)$ is  the ideal lattice of $R$. \end{thm}

\begin{proof}		
	
	Suppose that $I$ is an $n$-ideal of $R$. Let $J, K$ be any finitely generated ideals of $R$ such that $J \cdot K \leq I$ with $J \nleq \sqrt{0_L}$ in $L$. It is known that finitely generated ideals of $R$ are compact elements of $Id(R)$. Since $J \nleq \sqrt{0_L}$, we get $J^n \not = 0_L=(0_R)$ for every $n \in \mathbb{N}$.  Hence $J \cap (R \setminus \sqrt{0_R}) \not = \emptyset$. By Theorem \ref{6.11.}, $K \subseteq I$, i.e., $K \leq I$. Therefore $I$ is a $n$-element of $L$.

	Conversely, suppose that  $I$ is a $n$-element of $L$. Let $a, b \in R$ such that $a \cdot b \in I$ with $a \notin \sqrt{0_R}$. We claim that $b \in I$. Since  $(a \cdot b) = (a) \cdot (b) \subseteq I$, we have $a' \cdot b' \leq I$ in $L$, where $a' = (a),\  b' = (b) \in L_*$.  Clearly, $a' \nleq \sqrt{0_L}$. Hence $b' \leq I$, i.e., $b \in I$. Thus $I$ is an $n$-ideal of $R$.  \end{proof}

\begin{thm} \label{7.36} Let $R$ be a commutative ring with unity. Then $I$ is a $J$-ideal of $R$ if and only if $I$ is a $J$-element of multiplicative lattice $L = Id(R)$, where $Id(R)$ is  the ideal lattice of $R$. \end{thm}

\begin{proof}		
	
	Suppose that $I$ is a $J$-ideal of $R$. Let $A, B$ be finitely generated ideals of $R$ (which  are compact elements of $Id(R)$) such that $A \cdot B \leq I$ with $A \nleq J(L)$ in $L=Id(R)$. Since $A \nleq J(L)$, we get $A \nsubseteq J(R)$. By Theorem \ref{7.11.}, $B \subseteq I$, i.e., $B \leq I$ in $L$. Hence $I$ is a $J$-element of $L$.

	Conversely, suppose that  $I$ is a $J$-element of $L$. Let $a, b \in R$ such that $a \cdot b \in I$ with $a \notin J(R)$. We claim that $b \in I$. Since  $(a \cdot b) = (a) \cdot (b) \subseteq I$, we have $a' \cdot b' \leq I$ in $L$, where $a' = (a), \ b' = (b) \in L_*$.  Clearly, $a' \nleq J(L)$. Hence $b' \leq I$, i.e.,  $b \in I$. Hence $I$ is a $J$-ideal of $R$.  \end{proof}

Let $L$ be a multiplicative lattice. Then one can see the each of the sets $Z(L)$, $\bigl(\sqrt{0}\ \!\bigr]$ and $\bigl(J(L)\!\bigr]$ are multiplicatively  $M$-closed subsets of $L$. So if we replace $\mathfrak{X}$ by these sets, then we get the results of $r$-element, $n$-element and $J$-element respectively. 

\vskip5pt

We quote some of these results for ready reference.

\vskip 5truept 

One can see that in a $c$-lattice $L$,  $\bigl(L_* \setminus Z(L)\bigr)\subseteq \bigl(L_* \setminus \bigl(\sqrt{0}\ \!\bigr]\bigr)$ and  $\bigl(\sqrt{0}\ \!\bigr]  \subseteq \bigl(J(L)\!\bigr]$. For this, let $x \in \bigl(L_* \setminus Z(L)\bigr)$ and $x \in \bigl(\sqrt{0}\ \!\bigr]$. Then $x^n=0$ for some $n \in \mathbb{N}$. Thus $x \in Z(L)$, a contradiction. This proves the inclusion $\bigl(L_* \setminus Z(L)\bigr)\subseteq \bigl(L_* \setminus \bigl(\sqrt{0}\ \!\bigr]\bigr)$. Now, for the second inclusion, let $y$ be any compact element such that $y \in \bigl(\sqrt{0}\ \!\bigr] $. Then $y ^k=0$ for some $k \in 
\mathbb{N}$. Let $m$ be a maximal element of $L$. Then it is prime. This together with $y^k=0 \leq m$ implies that $y \leq m$. This further yields that $y \in J(L)=\bigwedge_{k \in \Lambda} m_k$. Since $L$ is a $c$-lattice and every compact element below $\sqrt{0}$ is below $J(L)$, we have $\sqrt{0} \leq J(L)$.

\vskip 5truept 

Hence by Lemma \ref{imply},  we have the following result.

\begin{prop}\label{relation}
	Let $L$ be a $c$-lattice.  Then every $n$-element of $L$   is a $r$-element as well as it is a $J$-element of $L$.
\end{prop}

By Proposition \ref{relation}, Theorems \ref{8.36}, \ref{2.36}, and Theorem \ref{7.36}, we have:

\begin{prop}
	Let $R$ be a commutative ring with unity. Then every $n$-ideal of $R$   is a $r$-ideal as well as it is a $J$-ideal.
\end{prop}

From Lemma \ref{4.10.}, we get Proposition 2.2 of \cite{KB} and Proposition 2.3 of \cite{TKO}.   Also,  Proposition 2.4 of \cite{TKO} follows from Lemma \ref{4.12.}. It is easy to observe that Proposition 2.10 of \cite{KB} and Theorem 2.7 of \cite{TKO} follows from Lemma \ref{4.13.}. We observe that Proposition 2.13 of \cite{KB} follows from Lemmas \ref{4.10.} and \ref{4.18.}. Note that Theorem \ref{2.20.} strengthens Theorem 2.12 of \cite{TKO}. Lemmas \ref{2.22.} and \ref{2.23.} generalizes the equivalence of $(i)$ and $(ii)$ in  Corollary 2.13 of \cite{TKO} and Proposition 2.20 of \cite{KB} respectively.  One can see that Lemma \ref{cancellation} extends Proposition 2.16 of \cite{TKO} and Proposition 2.21  of \cite{KB}.   Lastly, Theorem 2.23 of \cite{TKO} and Proposition 2.29 of \cite{KB} follows from Theorem \ref{2.30.}.

\vskip 5truept 
For the following result, we need a little more explanation.

\begin{prop}[{\cite[Proposition 2.3]{KB}}]
	Let $R$ be a commutative ring with unity. Then the following are equivalent.	
\end{prop}

\begin{enumerate}
	\item $R$ is a local ring.
	\item Every proper ideal of $R$ is a $J$-ideal.
	
\end{enumerate}

\begin{proof}
	$(1) \implies (2):$ It is clear that the ideal lattice $Id(R)$ of $R$ is a local lattice. Further, $J(L)=m$, where $m$ is the unique maximal element of $Id(R)$. Hence by Lemma \ref{local}, every proper element of $L$ is an $\mathfrak{X}$-element, where $\mathfrak{X}=(m]=\bigl(J(L)\!\bigr]$. That is, every proper element of $L$ is a $J$-element. By Theorem \ref{7.36}, every proper ideal of $R$ is a $J$-ideal.

	
	$(2) \implies (1):$ follows from Lemma \ref{meet-irr}
\end{proof}

Finally, the results of $n$-multiplicatively closed subset and $J$-multiplicatively closed subset can be obtained by using Lemma \ref{4.29.}. Further, the results of $r$-ideals can be deduced from our results for Noetherian rings, since Theorem \ref{8.36} is available for Noetherian lattice settings. If Question \ref{que} has an affirmative answer, then our results will extend most of the results of $r$-ideals of a commutative ring with unity.

\end{document}